\documentclass[letterpaper,11pt,twoside,keywordsasfootnote,addressatend,noinfoline]{article}
\usepackage{fullpage}
\usepackage[english]{babel}
\usepackage{amssymb}
\usepackage{amsmath}
\usepackage{theorem}
\usepackage{epsfig}
\usepackage{subfigure}
\usepackage{mathrsfs}
\usepackage{color}
\usepackage{imsart}

\newtheorem{theorem}{Theorem}
\newtheorem{lemma}[theorem]{Lemma}

\newenvironment{proof}{\noindent{\bf Proof.}}{\hspace*{2mm}~$\square$}

\newcommand{\N}{\mathbb{N}}
\newcommand{\Z}{\mathbb{Z}}
\newcommand{\G}{\mathscr{G}}
\newcommand{\V}{\mathscr{V}}
\newcommand{\E}{\mathscr{E}}
\newcommand{\ind}{\mathbf{1}}
\newcommand{\n}{\hspace*{-5pt}}

\DeclareMathOperator{\card}{card}
\DeclareMathOperator{\uniform}{Uniform \,}
\DeclareMathOperator{\geometric}{Geometric}

\begin{document}

\begin{frontmatter}
\title     {Distribution of money on connected graphs \\ with multiple banks}
\runtitle  {Distribution of money with multiple banks}
\author    {Nicolas Lanchier\thanks{Nicolas Lanchier was partially supported by NSF grant CNS-2000792.} and Stephanie Reed}
\runauthor {Nicolas Lanchier and Stephanie Reed}
\address   {School of Mathematical and Statistical Sciences \\ Arizona State University \\ Tempe, AZ 85287, USA. \\ nicolas.lanchier@asu.edu}
\address   {Mathematics Department \\ California State University \\ Fullerton, CA 92834, USA. \\ sjreed@fullerton.edu}

\maketitle

\begin{abstract} \ \
 This paper studies an interacting particle system of interest in econophysics inspired from a model introduced in the physics literature.
 The original model consists of the customers of a single bank characterized by their capital, and the discrete-time dynamics consists of monetary transactions in which a random individual~$x$ gives one coin to another random individual $y$, the transaction being canceled when $x$ is in debt and there is no more coins to borrow from the bank.
 Using a combination of numerical simulations and heuristic arguments, physicists conjectured that the distribution of money~(the distribution of the number of coins owned by a given individual) at equilibrium converges to an asymmetric Laplace distribution in the large population/temperature limit.
 In this paper, we prove and extend this conjecture to a more general model including multiple banks and interactions among customers across banks.
 More importantly, our model assumes that customers are located on a general undirected connected graph~(as opposed to the complete graph in the original model) where neighbors are interpreted as business partners, and transactions occur along the edges, thus modeling the flow of money across a social network.
 We show the convergence to the asymmetric Laplace distribution in the large population/temperature limit for any graph, thus proving and extending the conjecture from the physicists, and derive an exact expression of the distribution of money for all population sizes and money temperatures. \\
\end{abstract}



\end{frontmatter}

%%%%%%%%%%%%%%%%%%%%%%%%%%%%%%%%%%%%%%%%%%%%%%%%%%%%%%%%%%%%%%%%%%%%%%%%%%%%%%%%%%%%%%%%%%%%%%%%%%%%%%%%%%%%%%%%%%%%%%%%%%%%%%%%%%%%%%%%%%%%%%%%%%%%%%%%%%%%%%%%%%%%%%%%%%%%%%%%%%%%%%%%%%%%%%%%%%

\section{Introduction}
\label{sec:intro}
 The main objective of this paper is to prove (and extend) the conjecture from~\cite{xi_ding_wang_2005} about an interacting particle system of interest in econophysics, the sub field of statistical mechanics that applies traditional concepts from physics to model and study economical systems.
 Many such stochastic models are reviewed in~\cite{yakovenko_et_al_2009} and consist of large populations of individuals characterized by the amount of money they possess, which we identify as a \emph{number of coins}.
 In all these models, two individuals, say~$x$ and~$y$, are chosen uniformly at random and sequentially from the entire population at each time step to engage in a monetary transaction, and the models only differ in the exchange rules at each interaction.
 From the point of view of econophysics, the individuals can be thought of as particles, money as energy, and the mean number of coins per individual as the temperature, also referred to as the \emph{money temperature}. 
 The objective of research in this field is to determine the so-called \emph{distribution of money}, i.e., the fraction of individuals that have a given number of coins at equilibrium and in the large population and large money temperature limit.
 The exchange rules reviewed in~\cite{yakovenko_et_al_2009} along with the corresponding conjectures found by physicists based on numerical simulations and/or heuristic arguments are the following. \vspace*{5pt} \\
\noindent{\bf One-coin model}~\cite{dragulescu_yakovenko_2000}.
 In this model, individual~$x$ gives one coin to individual~$y$ if she has at least one coin, and it was conjectured that, in the large population/temperature limit, the distribution of money converges to the exponential distribution with mean the money temperature. \vspace*{5pt} \\
\noindent{\bf Uniform reshuffling model}~\cite{dragulescu_yakovenko_2000}.
 In this case, all the coins of individuals~$x$ and~$y$ are uniformly redistributed between the two interacting individuals, and it was conjectured that, in the large population/temperature limit, the distribution of money converges to the exponential distribution with mean the money temperature, just like in the one-coin model. \vspace*{5pt} \\
\noindent{\bf Immediate exchange model}~\cite{heinsalu_patriarca_2014, katriel_2015}.
 In this model, individuals~$x$ and~$y$ choose independently and uniformly a random number of their coins to give to the other individual, and it was conjectured that the distribution of money now converges to a gamma distribution with mean the money temperature and shape parameter two in the large population/temperature limit. \vspace*{5pt} \\
\noindent{\bf Saving propensity model}~\cite{chakraborti_chakrabarti_2000, patriarca_chakraborti_kaski_2004}.
 Under this rule, individuals~$x$ and~$y$ independently save a random number of their coins while the remaining coins are uniformly redistributed between the two agents, just like in the uniform reshuffling model.
 In this case, it was conjectured that the distribution of money converges to a gamma distribution with mean the money temperature and shape parameter two, just like in the immediate exchange model. \vspace*{5pt} \\
\noindent{\bf Individual debt model}~\cite{dragulescu_yakovenko_2000}.
 For this model, the rule at each interaction is the same as in the one-coin model except that individual~$x$ can now go into debt with an individual limit of~$L$ coins.
 In this case, it was conjectured that the distribution of money in the large population/temperature limit converges to a shifted exponential distribution shifted by~$- L$ coins. \vspace*{5pt} \\
\noindent{\bf Bank model}~\cite{xi_ding_wang_2005}.
 The rule of this model is again the same as in the one-coin model except that individual~$x$ can now borrow money from a bank provided the bank is not empty, in which case it was conjectured that, in the large population/temperature limit, the distribution of money converges to an asymmetric Laplace distribution. \vspace*{5pt} \\
\noindent Convergence of the one-coin model to the exponential distribution in the large population and large temperature limit was proved rigorously in~\cite{lanchier_2017b}.
 The authors also proved the conjectures about the uniform reshuffling, immediate exchange and saving propensity models in~\cite{lanchier_reed_2018} as well as the conjecture about the individual debt model in~\cite{lanchier_reed_2019}.
 While we also studied in~\cite{lanchier_reed_2019} some of the aspects of the bank model, we were unable to establish the convergence to the asymmetric Laplace distribution conjectured in~\cite{xi_ding_wang_2005}.
 In this paper, we prove this conjecture but also extend the result to a more general and realistic model that includes multiple banks and an explicit spatial structure.
 More precisely, the individuals in our model are now located on a general connected graph and pairs of individuals can only interact~(exchange money) if they are connected by an edge, which contrasts with the six models above that were studied by physicists via numerical simulations on the complete graph where individuals interact globally.
 This underlying graph has to be thought of as a social network where edges represent links between business partners.
 
%%%%%%%%%%%%%%%%%%%%%%%%%%%%%%%%%%%%%%%%%%%%%%%%%%%%%%%%%%%%%%%%%%%%%%%%%%%%%%%%%%%%%%%%%%%%%%%%%%%%%%%%%%%%%%%%%%%%%%%%%%%%%%%%%%%%%%%%%%%%%%%%%%%%%%%%%%%%%%%%%%%%%%%%%%%%%%%%%%%%%%%%%%%%%%%%%%

\section{Model description and main results}
\label{sec:model}
 To define the model with multiple banks, we start by letting~$\G = (\V, \E)$ be a general finite connected graph.
 This graph represents a social network where each vertex~$x \in \V$ is interpreted as an individual while an edge~$(x, y) \in \E$ indicates that~$x$ and~$y$ are business partners.
 The topology of the graph is incorporated in the dynamics of the economical system by assuming that the flow of money can only occurs through the edges, meaning that only business partners can interact to exchange money.
 We also assume that the system includes~$K$ banks and that each individual is a customer at exactly one of the banks:
 let~$\V_1, \V_2, \ldots, \V_K$ be a partition of the vertex set and assume that each~$x \in \V_i$ is a customer at bank~$i$ only.
 Individuals or vertices are characterized by the number of coins they possess, which is negative for individuals in debt who borrowed money from their bank.
 More precisely, the state at time~$t \in \N$ is a spatial configuration
 $$ \xi_t : \V \to \Z \quad \hbox{where} \quad \xi_t (x) = \left\{\begin{array}{l}
        + \ \hbox{number of coins~$x$ possesses when~$\xi_t (x) \geq 0$} \vspace*{2pt} \\
        - \ \hbox{number of coins~$x$ borrowed  when~$\xi_t (x) \leq 0$}. \end{array} \right. $$
 Let~$M$ be the number of coins in the population,~$N_i$ be the number of customers of bank~$i$, and~$R_i$ be the number of coins in bank~$i$ initially, and
 $$ N = N_1 + N_2 + \cdots + N_K \quad \hbox{and} \quad R = R_1 + R_2 + \cdots + R_K $$
 be respectively the number of vertices (population size) and the total number of coins in all the banks initially.
 From now on, we also let
 $$ T = M / N = \hbox{money temperature} \quad \hbox{and} \quad \rho_i = R_i / M \quad \hbox{for} \quad i = 1, 2, \ldots, K. $$
 At each time step, we choose two neighbors, say~$x$ and~$y$, sequentially uniformly at random, or equivalently an oriented edge~$\vec{xy}$ uniformly at random, and think of~$x$ as a potential buyer and~$y$ as a potential seller.
 Assuming that~$x \in \V_i$ and~$y \in \V_j$, we move one coin~$x \to y$ if there is
\begin{itemize}
 \item at least one coin at vertex~$x$ or \vspace*{2pt}
 \item less than one coin at vertex~$x$ but at least one coin in bank~$i$ that~$x$ can borrow.
\end{itemize}
 Otherwise, the buyer~$x$ is unable to pay and cannot borrow money so the transaction is canceled and nothing happens.
 In case a transfer of money does happen, the seller~$y$ either
\begin{itemize}
 \item increases her capital by one coin or \vspace*{2pt}
 \item decreases her debt by one coin by returning the coin to bank~$j$.
\end{itemize}
 To study the distribution of money, we first prove that the process~$(\xi_t)$ is ergodic and reversible, and that its unique stationary distribution is the uniform distribution on the set of admissible configurations, i.e., the configurations such that the total fortune of all the individuals equals~$M$ and the total debt of all the customers of bank~$i$ does not exceed~$R_i$.
 In particular, computing the probability that an individual, say~$x$, possesses~$c$ coins at equilibrium reduces to counting the total number of configurations and the number of configurations with exactly~$c$ coins at vertex~$x$.
 Using a combinatorial argument, we can prove the following general result which holds for all possible choices of~$M, N_i$ and~$R_i$, not just in the large population/temperature limit.
\begin{theorem} --
\label{th:combinatorics}
 Let~$e_j$ be the~$j$th unit vector in~$\Z^K$.
 For all connected graph~$\G$, all~$M, N_i$ and~$R_i$, and all~$x \in \V_j$, the probability that vertex~$x$ has~$c$ coins at equilibrium is given by
 $$ \begin{array}{ccl}
    \displaystyle \frac{\Lambda ((N_1, \ldots, N_K) - e_j, (R_1, \ldots, R_K), M - c)}{\Lambda ((N_1, \ldots, N_K), (R_1, \ldots, R_K), M)} & \quad
    \hbox{for all} & 0 \leq c \leq M \vspace*{8pt} \\
    \displaystyle \frac{\Lambda ((N_1, \ldots, N_K) - e_j, (R_1, \ldots, R_K) + c e_j, M - c)}{\Lambda ((N_1, \ldots, N_K), (R_1, \ldots, R_K), M)} & \quad
    \hbox{for all} & 0 \geq c \geq - R_j \end{array} $$
 where the function~$\Lambda : \N^K \times \N^K \times \N$ is defined as
 $$ \begin{array}{rcl}
    \Lambda (N_i, R_i, M) & \n = \n &
    \displaystyle \sum_{a_1 = 0}^{R_1} \cdots \sum_{a_K = 0}^{R_K} \ \sum_{b_1 = 0}^{N_1} \cdots \sum_{b_K = 0}^{N_K}
    \displaystyle {N_1 \choose b_1} \cdots {N_K \choose b_K} \vspace*{8pt} \\ && \hspace*{25pt}
    \displaystyle {a_1 - 1 \choose b_1 - 1} \cdots {a_K - 1 \choose b_K - 1} {M + a_1 + \cdots + a_K + N - b_1 - \cdots - b_K - 1 \choose N - b_1 - \cdots - b_K - 1}. \end{array} $$
\end{theorem}
\begin{figure}[t]
\centering
\scalebox{0.32}{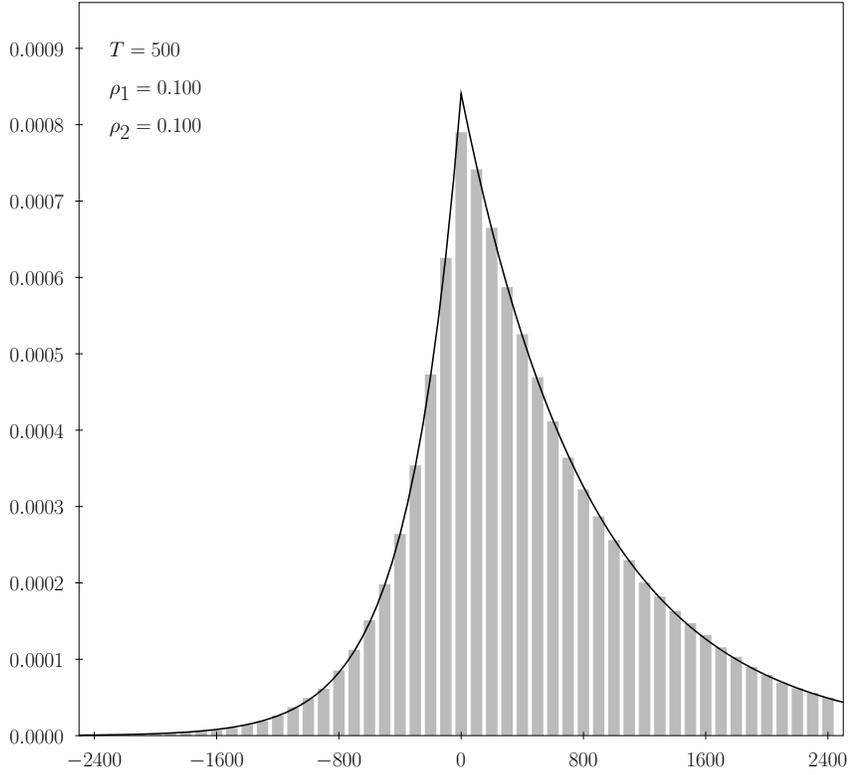}
\caption{\upshape{Distribution of money for the model with two identical banks (with the same number of customers and the same initial number of coins).
 The grey histogram is obtained from numerical simulation of the stochastic process on the complete graph with~$N = 10$K vertices, $M = 5$M coins, and~$R = 1$M coins initially in the two banks.
 The solid curve represents the density function~$f$ found in Theorem~\ref{th:laplace}.}}
\label{fig:symmetry}
\end{figure}
\begin{figure}[t]
\centering
\scalebox{0.19}{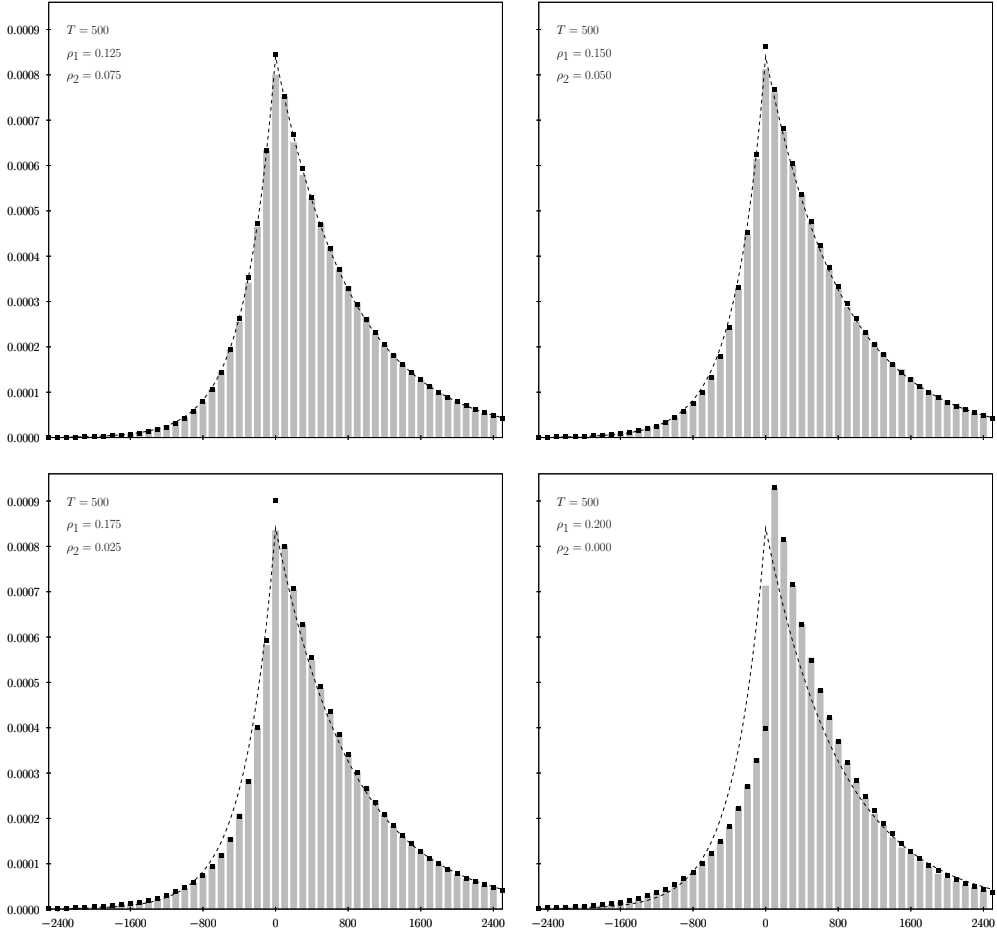}
\caption{\upshape{Distribution of money for the model with two banks with the same number of customers but different initial numbers of coins ($\rho_1 \neq \rho_2$).
 The grey histogram is again obtained from numerical simulation of the stochastic process on the complete graph with~$N = 10$K vertices, $M = 5$M coins, and~$R = 1$M coins in the banks.
 The dashed curve is just a duplicate of the solid curve in Figure~\ref{fig:symmetry} displayed for comparison.
 The black squares are computed from combining~Theorem~\ref{th:combinatorics} and~\eqref{eq:convex} for~$N = 100$, $M = 50$K and~$R = 10$K.}}
\label{fig:asymmetry}
\end{figure}
 Note that, when the initial number of coins and/or the number of customers in each bank are different, the probability that a given individual has~$c$ coins depends on the bank she goes to.
 In particular, the distribution of money consists of a convex combination of the probabilities given in the previous theorem.
 More precisely, letting~$x_1 \in \V_1, \ldots, x_K \in \V_K$, the fraction of individuals with~$c$ coins at equilibrium is given by
\begin{equation}
\label{eq:convex}
  \begin{array}{rcl} f (c) & \n = \n &
  \displaystyle \frac{1}{N} \,\sum_{x \in \V} \ \lim_{t \to \infty} P (\xi_t (x) = c) =
  \displaystyle \sum_{j = 1}^K \ \frac{N_j}{N} \ \lim_{t \to \infty} P (\xi_t (x_j) = c) \vspace*{8pt} \\ & \n = \n &
  \displaystyle \sum_{j = 1}^K \ \frac{N_j}{N} \
  \displaystyle \frac{\Lambda ((N_1, \ldots, N_K) - e_j, (R_1, \ldots, R_K), M - c)}{\Lambda ((N_1, \ldots, N_K), (R_1, \ldots, R_K), M)} \
  \ind \{0 \leq c \leq M \} \vspace*{8pt} \\ && \hspace*{10pt} + \ 
  \displaystyle \sum_{j = 1}^K \ \frac{N_j}{N} \
  \displaystyle \frac{\Lambda ((N_1, \ldots, N_K) - e_j, (R_1, \ldots, R_K) + c e_j, M - c)}{\Lambda ((N_1, \ldots, N_K), (R_1, \ldots, R_K), M)} \
  \ind \{0 > c \geq - R_j \}. \end{array}
\end{equation}
 Looking now at the distribution of money in the large population/temperature limit, one natural approach to prove the conjecture in~\cite{xi_ding_wang_2005} is to study the right-hand side of equation~\eqref{eq:convex} when both the population size~$N$ and the money temperature~$T = M/N$ are large.
 However, we were not able to simplify this expression enough to deduce convergence to an asymmetric Laplace distribution so we use a completely different approach.
 Because the convergence of the process~$(\xi_t)$ to the uniform distribution on the set of configurations holds \emph{regardless} of the choice of the connected graph, we may focus on the complete graph, in which case the limiting behavior for large populations is well approximated by a system of coupled differential equations.
 The analysis of the fixed point(s) of this system shows convergence to an asymmetric Laplace distribution.
 To find the value of the three parameters of the Laplace distribution, we use the conservation of the money temperature and prove that, in the symmetric case where all the banks have the same number of customers and start with the same number of coins, the total number of coins in the banks in the large population limit becomes much smaller than~$R$.
 More precisely, we have the following theorem that extends the conjecture in~\cite{xi_ding_wang_2005} to multiple banks and general connected graphs.
\begin{theorem} --
\label{th:laplace}
 Assume~$N_i = N / K$ and~$R_i = R / K$.
 When~$N$ and~$T = M / N$ are large,
\begin{equation}
\label{eq:Laplace-1}
  \lim_{t \to \infty} P (\xi_t (x) = c) \approx f (c) = \left\{\hspace*{-3pt} \begin{array}{lcl} \mu \,e^{- ac} & \hbox{for} & c \geq 0 \vspace*{3pt} \\
                                                                                                 \mu \,e^{+ bc}  & \hbox{for} & c \leq 0 \end{array} \right.
\end{equation}
 where, letting~$\rho = R/M = (R_1 + \cdots + R_K) / M$,
\begin{equation}
\label{eq:Laplace-2}
  \mu = \frac{1}{T} \bigg(\sqrt{1 + \rho} - \sqrt{\rho} \bigg)^2, \qquad
    a = \frac{1}{T} \bigg(1 - \sqrt{\frac{\rho}{1 + \rho}} \bigg), \qquad
    b = \frac{1}{T} \bigg(\sqrt{\frac{1 + \rho}{\rho}} - 1 \bigg).
\end{equation}
\end{theorem}
 For a picture of the Laplace distribution~\eqref{eq:Laplace-1} with parameters~\eqref{eq:Laplace-2} along with the distribution of money obtained from numerical simulation of the stochastic process on a large complete graph in the presence of two identical banks, we refer the reader to Figure~\ref{fig:symmetry}.
 We also used a computer program in Figure~\ref{fig:asymmetry} to display the distribution of money obtained from Theorem~\ref{th:combinatorics} and equation~\eqref{eq:convex} in the presence of two non-identical banks along with simulation of the process.
 The money temperature and the fraction~$\rho_1 + \rho_2$ of coins initially in the two banks are the same as in Figure~\ref{fig:symmetry} and the banks again have the same number of customers but the banks now start with different numbers of coins, meaning that~$\rho_1 \neq \rho_2$.
 The pictures show that, as the number of coins in the first bank increases and the number of coins in the second bank decreases, the number of individuals in debt at equilibrium decreases.
 In addition, the value of~$c$ at which the maximum of the distribution of money is reached shifts to the right, suggesting that the convergence to an asymmetric Laplace distribution in Theorem~\ref{th:laplace} does not hold in the presence of non-identical banks. 

%%%%%%%%%%%%%%%%%%%%%%%%%%%%%%%%%%%%%%%%%%%%%%%%%%%%%%%%%%%%%%%%%%%%%%%%%%%%%%%%%%%%%%%%%%%%%%%%%%%%%%%%%%%%%%%%%%%%%%%%%%%%%%%%%%%%%%%%%%%%%%%%%%%%%%%%%%%%%%%%%%%%%%%%%%%%%%%%%%%%%%%%%%%%%%%%%%

\section{Convergence to the uniform distribution}
\label{sec:uniform}
 This section is devoted to collecting some key preliminary results that will be used later to prove the theorems.
 To begin with, we prove that the process is irreducible~(see Lemma~\ref{lem:irreducibility}) and aperiodic~(see Lemma~\ref{lem:aperiodicity}), and so converges to a unique stationary distribution~$\pi$ that does not depend on the initial configuration of coins.
 Note that this distribution is a probability measure on the set of configurations as opposed to the distribution of money which is defined as the distribution of the number of coins at a given vertex.
 Then, using reversibility, we deduce that the distribution~$\pi$ is uniform on the set of configurations~(see Lemma~\ref{lem:uniform}).
 First, we give an explicit expression of the transition probabilities of the process.
 This is done in the next lemma where~$\tau_{xy}$ is the operator on the set of configurations that moves one coin from vertex~$x$ to vertex~$y$, i.e.,
 $$ (\tau_{xy} \,\xi)(z) = (\xi (x) - 1) \,\ind \{z = x \} + (\xi (y) + 1) \,\ind \{z = y \} + \xi (z) \,\ind \{z \neq x, y \}, $$
 and where~$S$ denotes to the set of configurations, i.e.,
 $$ \begin{array}{l}
      S = \displaystyle \Bigg\{\xi \in \Z^{\V} : - R_i \leq \xi (x) \leq M + R \ \hbox{for all} \ x \in \V_i, \ \sum_{z \in \V} \,\xi (z) = M \vspace*{-4pt} \\ \hspace*{100pt}
          \displaystyle \hbox{and} \ \sum_{z \in \V_i} \,\xi (z) \,\ind \{\xi (z) < 0 \} \geq - R_i \ \hbox{for all} \ i = 1, 2, \ldots, K \Bigg\}. \end{array} $$
\begin{lemma} --
\label{lem:transition}
 For all~$(x, y) \in \V_i \times \V$,
 $$ p (\xi, \tau_{xy} \,\xi) = \frac{\ind \,\{(x, y) \in \E \}}{2 \card (\E)} \ \ \ind \bigg\{\xi (x) > 0 \ \hbox{or} \ \sum_{z \in \V_i} \,\xi (z) \,\ind \{\xi (z) < 0 \} > - R_i \bigg\}. $$
\end{lemma}
\begin{proof}
 Let~$\xi \in S$ and~$(x, y) \in \V_i \times \V$.
 Then, $p (\xi, \tau_{xy} \,\xi) > 0$ if and only if
\begin{enumerate}
 \item the individual at~$x$ is not in debt or bank~$i$ is not empty, and \vspace*{2pt}
 \item vertices~$x$ and~$y$ are connected by an edge.
\end{enumerate}
 These conditions can be expressed mathematically as
\begin{equation}
\label{eq:transition-1}
  \bigg(\xi (x) > 0 \ \hbox{or} \ \sum_{z \in \V_i} \,\xi (z) \,\ind \{\xi (z) < 0 \} > - R_i \bigg) \quad \hbox{and} \quad (x, y) \in \E.
\end{equation}
 In addition, given that the conditions in~\eqref{eq:transition-1} hold, the probability of transitioning from configuration~$\xi$ to configuration~$\tau_{xy} \,\xi$ is the probability that the directed edge~$\vec{xy}$ is selected.
 Since each undirected edge results in two directed edges, this probability is equal to
\begin{equation}
\label{eq:transition-2}
  \frac{1}{\card \{x, y \in \V : (x, y) \in \E \}} = \frac{1}{2 \card (\E)}.
\end{equation}
 The result then follows from combining~\eqref{eq:transition-1} and~\eqref{eq:transition-2}.
\end{proof} \\ \\
 In order to prove irreducibility, we first use Lemma~\ref{lem:transition} to prove that if it is possible to move a coin through the directed edge~$\vec{xy}$ in one time step then, right after this move, it is possible to move a coin from vertex~$y$ to any of its neighbors~$z$, again in one time step.
\begin{lemma} --
\label{lem:move-edge}
 Let~$\xi \in S$ and~$(x, y) \in \E$ such that~$\xi' = \tau_{xy} \,\xi \in S$. Then,
 $$ p (\xi, \xi') > 0 \quad \hbox{implies that} \quad \tau_{yz} \,\xi' \in S \quad \hbox{and} \quad p (\xi', \tau_{yz} \,\xi') > 0. $$
\end{lemma}
\begin{proof}
 Letting~$x \in \V_i$ and~$y \in \V_j$, we prove the result by distinguishing two cases depending on whether the two vertices go to the same bank or not. \vspace*{5pt} \\
\noindent {\bf Case 1}. Different banks~$i \neq j$. \vspace*{5pt} \\
 In this case, after the first move, there will be one more coin in bank~$j$ in case~$y$ was in debt therefore~$y$ can borrow a coin from this bank if needed.
 We now turn the intuition into rigorous equations.
 According to Lemma~\ref{lem:transition}, it suffices to prove that
\begin{equation}
\label{eq:move-edge-1}
  \xi' (y) > 0 \qquad \hbox{or} \qquad \sum_{z \in \V_j} \,\xi' (z) \,\ind \{\xi' (z) < 0 \} > - R_j.
\end{equation}
 Assuming that~$\xi' (y) \leq 0$, we must have
 $$ \xi' (y) \,\ind \{\xi' (y) < 0 \} = \xi' (y) \,\ind \{\xi' (y) \leq 0 \} = \xi' (y) = \xi (y) + 1. $$
 Using also that~$\xi' \in S$, we deduce that
 $$ \begin{array}{l}
    \displaystyle \sum_{z \in \V_j} \,\xi' (z) \,\ind \{\xi' (z) < 0 \} =
    \displaystyle \sum_{z \in \V_j \setminus \{y \}} \,\xi (z) \,\ind \{\xi (z) < 0 \} + \xi (y) + 1 \vspace*{5pt} \\ \hspace*{60pt} \geq
    \displaystyle \sum_{z \in \V_j} \,\xi (z) \,\ind \{\xi (z) < 0 \} + 1 \geq - R_j + 1 > - R_j, \end{array} $$
 which shows that~\eqref{eq:move-edge-1} holds in the first case. \vspace*{5pt} \\
\noindent {\bf Case 2}. Same bank~$i = j$. \vspace*{5pt} \\
 In this case, either~$x$ is not in debt before the first move which brings one coin to either~$y$ or the common bank~(in either case~$y$ can then use this coin) or~$x$ is in debt before the move but the fact that the move occurs indicates that the bank is not empty.
 To turn this into a rigorous proof, we again assume that~$\xi' (y) \leq 0$, and observe that, according to Lemma~\ref{lem:transition},
\begin{equation}
\label{eq:move-edge-2}
  \xi (x) > 0 \qquad \hbox{or} \qquad \sum_{z \in \V_i} \,\xi (z) \,\ind \{\xi (z) < 0 \} > - R_i
\end{equation}
 because~ $p (\xi, \tau_{xy} \,\xi) > 0$.
 This leads to two sub-cases: \vspace*{5pt} \\
\noindent {\bf Sub-case 2a}. Assume that~$\xi (x) > 0$. \vspace*{5pt} \\
 The same reasoning as in the first case shows that the second inequality in~\eqref{eq:move-edge-1} holds. \vspace*{5pt} \\
\noindent {\bf Sub-case 2b}. Assume that~$\xi (x) \leq 0$. \vspace*{5pt} \\
 In this case, the second inequality in~\eqref{eq:move-edge-2} must hold and
 $$ \begin{array}{rcl}
    \xi' (x) \,\ind \{\xi' (x) < 0 \} & \n = \n & \xi' (x) \,\ind \{\xi (x) - 1 < 0 \} = \xi' (x) = \xi (x) - 1, \vspace*{4pt} \\
    \xi' (y) \,\ind \{\xi' (y) < 0 \} & \n = \n & \xi' (y) \,\ind \{\xi' (y) \leq 0 \} = \xi' (y) = \xi (y) + 1, \end{array} $$
 from which it follows that
 $$ \begin{array}{l}
    \displaystyle \sum_{z \in \V_i} \,\xi' (z) \,\ind \{\xi' (z) < 0 \} =
    \displaystyle \sum_{z \in \V_i \setminus \{x, y \}} \,\xi (z) \,\ind \{\xi (z) < 0 \} + (\xi (x) - 1) + (\xi (y) + 1) \vspace*{5pt} \\ \hspace*{60pt} \geq
    \displaystyle \sum_{z \in \V_i} \,\xi (z) \,\ind \{\xi (z) < 0 \} > - R_i. \end{array} $$
 This again shows that~\eqref{eq:move-edge-1} holds.
\end{proof} \\ \\
 We now use Lemma~\ref{lem:move-edge} to prove that if two configurations~$\xi$ and~$\xi'$ can be obtained from one another by moving one coin from~$x$ to~$y$ then the configurations communicate.
 In addition, the minimum number of time steps required is less than twice the number of vertices.
\begin{lemma} --
\label{lem:move-path}
 Let~$\xi \in S$ and~$x, y \in \V$ such that~$\xi' = \tau_{xy} \,\xi \in S$. Then,
 $$ p_t (\xi, \xi') > 0 \quad \hbox{for some} \quad t < 2N. $$
\end{lemma}
\begin{proof}
 Because the graph is connected, there exists a self-avoiding directed path
 $$ x = x_0 \to x_1 \to \cdots \to x_t = y \quad \hbox{for some} \quad t < N $$
 going from vertex~$x$ to vertex~$y$.
 According to Lemma~\ref{lem:transition}, we can move a coin~$x = x_0 \to x_1$ whenever vertex~$x$ or the bank vertex~$x$ goes to has at least one coin:
\begin{equation}
\label{eq:move-path-1}
  \xi (x) > 0 \qquad \hbox{or} \qquad \sum_{z \in \V_i} \,\xi (z) \,\ind \{\xi (z) < 0 \} > - R_i
\end{equation}
 in which case, after applying Lemma~\ref{lem:move-edge} repeatedly, we can move the coin along the path up to vertex~$y$.
 In particular, to prove the lemma, it suffices to prove that~\eqref{eq:move-path-1} holds.
 It turns out, however, that this is the case for most but not all configurations~$\xi \in S$ such that~$\tau_{xy} \,\xi \in S$.
 The trouble appears when vertex~$x$ has no coin, vertex~$y$ is in debt, both vertices go to the same bank, and their bank is empty, in which case we will need to bring a coin from another vertex~$w$.
 We now turn the argument into a rigorous proof.
 To begin with, assume that~$x \in \V_i$ and~$y \in \V_j$. \vspace*{5pt} \\
\noindent {\bf Case 1}. Vertex~$x$ has at least one coin: $\xi (x) > 0$. \vspace*{5pt} \\
 In this case, \eqref{eq:move-path-1} holds so the result follows from the reasoning above using Lemmas~\ref{lem:transition} and~\ref{lem:move-edge}. \vspace*{5pt} \\
\noindent {\bf Case 2}. Vertex~$x$ does not have any coin: $\xi (x) \leq 0$. \vspace*{5pt} \\
 This case is more complicated so we consider several sub-cases depending on whether~$y$ is in debt or not and on whether the two vertices go to the same bank or not. \vspace*{5pt} \\
\noindent {\bf Sub-case 2a}. Vertex~$y$ in not in debt: $\xi (y) \geq 0$. \vspace*{5pt} \\
 In this case, $\xi' (x) = \xi (x) - 1 < 0$ and~$\xi' (y) = \xi (y) + 1 > 0$ therefore
 $$ \xi' (x) \,\ind \{\xi' (x) < 0 \} = \xi (x) - 1 \quad \hbox{and} \quad \xi' (y) \,\ind \{\xi' (y) < 0 \} = 0. $$
 Using also that~$\xi' = \tau_{xy} \,\xi \in S$, we deduce that
\begin{equation}
\label{eq:move-path-2}
  \begin{array}{l}
  \displaystyle \sum_{z \in \V_i} \,\xi (z) \,\ind \{\xi (z) < 0 \} >
  \displaystyle \sum_{z \in \V_i \setminus \{x \}} \xi (z) \,\ind \{\xi (z) < 0 \} + \xi (x) - 1 \vspace*{5pt} \\ \hspace*{60pt} =
  \displaystyle \sum_{z \in \V_i \setminus \{x \}} \xi' (z) \,\ind \{\xi' (z) < 0 \} \geq - R_i. \end{array}
\end{equation}
 In particular, \eqref{eq:move-path-1} holds so we can move a coin~$x \to y$ in less than~$N$ time steps. \vspace*{5pt} \\
\noindent {\bf Sub-case 2b}. The vertices go to different banks: $i \neq j$. \vspace*{5pt} \\
 In this case, we again have~$\xi' (x) \,\ind \{\xi' (x) < 0 \} = \xi (x) - 1$.
 In addition, $y \notin \V_i$ therefore~\eqref{eq:move-path-2} and so~\eqref{eq:move-path-1} again hold, showing that we can move a coin~$x \to y$. \vspace*{5pt} \\
\noindent {\bf Sub-case 2c}. Everything else: $\xi (y) < 0$ and~$i = j$. \vspace*{5pt} \\
 In contrast with 2a and 2b, the common bank might be empty in this case and so~\eqref{eq:move-path-2} might not hold.
 To deal with this problem, the idea is to find another vertex~$w$ that has at least one coin, move this coin along a path~$w \to y$, which will bring one coin in the bank that we can then move along a path~$x \to y$.
 To begin with, observe that, because
 $$ \sum_{z \in \V} \,\xi (z) = M > 0, \qquad \xi (x) \leq 0 \qquad \hbox{and} \qquad \xi (y) < 0, $$
 there exists~$w \in \V$, $w \neq x, y$, such that~$\xi (w) > 0$.
 Using again that the graph is connected, there exist two self-avoiding directed paths
 $$ x = x_0 \to x_1 \to \cdots \to x_s = w \qquad \hbox{and} \qquad w = y_0 \to y_1 \to \cdots \to y_t = y $$
 for some~$s, t < N$, going from vertex~$x$ to vertex~$w$, and from vertex~$w$ to vertex~$y$.
 Using as previously Lemmas~\ref{lem:transition} and~\ref{lem:move-edge}, and the fact that~$\xi (w) > 0$, we can move a coin along the path~$w \to y$, after which the configuration is~$\xi'' = \tau_{wy} \,\xi$.
 In this configuration, we have
 $$ \xi'' (x) \,\ind \{\xi'' (x) < 0 \} = \xi (x) \,\ind \{\xi (x) < 0 \} \quad \hbox{and} \quad \xi'' (y) \,\ind \{\xi'' (y) < 0 \} = \xi (y) + 1 $$
 from which it follows that
 $$ \sum_{z \in \V_i} \,\xi'' (z) \,\ind \{\xi'' (z) < 0 \} = \sum_{z \in \V_i} \,\xi (z) \,\ind \{\xi (z) < 0 \} + 1 > - R_i. $$
 In particular, the second condition in~\eqref{eq:move-path-1} holds for configuration~$\xi''$ so we can move a coin along the directed path~$x \to w$, after which the configuration is
 $$ \tau_{xw} \,\xi'' = \tau_{xw} \,\tau_{wy} \,\xi = \tau_{xy} \,\xi = \xi'. $$
 In conclusion, $p_{s + t} (\xi, \xi') > 0$ for some~$s, t < N$, and the proof is complete.
\end{proof} \\ \\
 Using Lemma~\ref{lem:move-path}, we can now deduce that the process is irreducible.
\begin{lemma}[irreducibility] --
\label{lem:irreducibility}
 The process~$(\xi_t)$ is irreducible.
\end{lemma}
\begin{proof}
 The idea is simple:
 given two configurations~$\xi, \xi' \in S$, we can go from one configuration to the other one by moving coins around, which can be done in a finite number of time steps and with positive probability according to Lemma~\ref{lem:move-path}.
 To make this statement rigorous, define a distance~$D$ on the set of configurations by letting
 $$ D (\xi, \xi') = \sum_{z \in \V} \,|\xi (z) - \xi' (z)| \quad \hbox{for all} \quad \xi, \xi' \in S. $$
 As long as~$\xi \neq \xi'$, because the number of coins is preserved by the dynamics and so the coordinates of each of the two configurations add up to~$M$,
 $$ \xi (x) > \xi' (x) \quad \hbox{and} \quad \xi (y) < \xi' (y) \quad \hbox{for some} \quad x, y \in \V. $$
 Then, moving one coin~$x \to y$ when in configuration~$\xi$ reduces the distance to~$\xi'$. Indeed,
 $$ \begin{array}{l}
    \displaystyle D (\tau_{xy} \,\xi, \xi') =
    \displaystyle \sum_{z \in \V} \,|\tau_{xy} \,\xi (z) - \xi' (z)| \vspace*{4pt} \\ \hspace*{20pt} =
    \displaystyle \sum_{z \in \V \setminus \{x, y \}} |\xi (z) - \xi' (z)| + |\xi (x) - 1 - \xi' (x)| + |\xi (y) + 1 - \xi' (y)| \vspace*{4pt} \\ \hspace*{80pt} =
    \displaystyle \sum_{z \in \V} \,|\xi (z) - \xi' (z)| - 2 = D (\xi, \xi') - 2. \end{array} $$
 This shows that, starting from configuration~$\xi$, we can move the coins around in such a way that, after a finite number of time steps and with positive probability, the new configuration is distance zero from~$\xi'$ and so equal to~$\xi'$.
 In particular, the process is irreducible.
\end{proof} \\ \\
 Aperiodicity directly follows from Lemma~\ref{lem:transition}.
\begin{lemma}[aperiodicity] --
\label{lem:aperiodicity}
 The process~$(\xi_t)$ is aperiodic.
\end{lemma}
\begin{proof}
 Because the process is irreducible, it suffices to find one configuration with period one.
 To do this, fix~$x \in \V$ and let~$\xi$ be any configuration such that
 $$ \xi (x) \leq 0 \qquad \hbox{and} \qquad \sum_{z \in \V_i} \,\xi (z) \,\ind \{\xi (z) < 0 \} = - R_i. $$
 Then, according to Lemma~\ref{lem:transition},
 $$ p (\xi, \tau_{xy} \,\xi) = 0 \quad \hbox{for all} \quad (x, y) \in \E $$
 from which it follows that whenever the directed edge chosen uniformly at random starts from vertex~$x$, the process stays in configuration~$\xi$.
 This implies that
 $$ p (\xi, \xi) \geq \frac{\deg (x)}{2 \card (\E)} \geq \frac{1}{2 \card (\E)} > 0 $$
 which proves that~$\xi$ has period one.
\end{proof} \\ \\
 Irreducibility and aperiodicity, together with the fact that the state space~$S$ is finite, imply that the process has a unique stationary distribution~$\pi$ to which the process converges regardless of the initial configuration.
 Using reversibility~(or alternatively, double stochasticity), we now prove that~$\pi$ is the uniform distribution on the set of configurations.
\begin{lemma}[reversibility] --
\label{lem:uniform}
 The stationary distribution~$\pi = \uniform (S)$.
\end{lemma}
\begin{proof}
 To prove reversibility and double stochasticity, it suffices to prove that
\begin{equation}
\label{eq:uniform-1}
  p (\xi, \xi') = p (\xi', \xi) \quad \hbox{for all} \quad \xi, \xi' \in S.
\end{equation}
 The result is obvious when~$\xi' = \xi$.
 When~$\xi' \neq \tau_{xy} \,\xi$ for all~$(x, y) \in \E$, it follows from Lemma~\ref{lem:transition} that the two transition probabilities in~\eqref{eq:uniform-1} are equal to zero so the result is again true in this case.
 To deal with the remaining cases, assume that
 $$ \xi' = \tau_{xy} \,\xi \quad \hbox{for some} \quad (x, y) \in \E. $$
 Then, applying Lemma~\ref{lem:move-edge} with~$z = x$, we get
\begin{equation}
\label{eq:uniform-2}
  p (\xi, \xi') = p (\xi, \tau_{xy} \,\xi) > 0 \quad \hbox{implies that} \quad p (\xi', \xi) = p (\tau_{xy} \,\xi, \tau_{yx} \,\tau_{xy} \,\xi) > 0.
\end{equation}
 Since in addition the transition probabilities are equal to either zero or~$1 / 2 \card (E)$ according to Lemma~\ref{lem:transition}, the two probabilities in~\eqref{eq:uniform-2} must be equal, which proves~\eqref{eq:uniform-1}.
 In particular, $\uniform (S)$ is a stationary distribution so the lemma follows from the uniqueness of~$\pi$.
\end{proof}

%%%%%%%%%%%%%%%%%%%%%%%%%%%%%%%%%%%%%%%%%%%%%%%%%%%%%%%%%%%%%%%%%%%%%%%%%%%%%%%%%%%%%%%%%%%%%%%%%%%%%%%%%%%%%%%%%%%%%%%%%%%%%%%%%%%%%%%%%%%%%%%%%%%%%%%%%%%%%%%%%%%%%%%%%%%%%%%%%%%%%%%%%%%%%%%%%%

\section{Number of configurations of coins}
\label{sec:combinatorics}
 Because the stationary distribution is~$\uniform (S)$, meaning that all the configuration are equally likely, the distribution of money at equilibrium, i.e., the probability that a given vertex has~$c$ coins under the distribution~$\pi$, is equal to the number of configurations with exactly~$c$ coins at that vertex divided by the total number of configurations.
 This section is devoted to counting those configurations, from which Theorem~\ref{th:combinatorics} follows.
 To begin with, we compute the number of configurations in the absence of banks, i.e., the special case~$R = 0$.
\begin{lemma} --
\label{lem:combination}
 Assume that~$R = 0$. Then the number of configurations is
 $$ \Lambda ((N_1, \ldots, N_K), (R_1, \ldots, R_K), M) = \Lambda (N_i, 0, M) = {M + N - 1 \choose N - 1}. $$
\end{lemma}
\begin{proof}
 By definition, the number of configurations~$\Lambda (N_i, R_i, M)$ is the number of integer solutions~$\xi (x)$, $x \in \V$, to the equation
 $$ \sum_{x \in \V} \,\xi (x) = M \quad \hbox{such that} \quad \sum_{y \in \V_i} \,\xi (y) \,\ind \{\xi (y) < 0 \} \geq - R_i \quad \hbox{for all} \quad i. $$
 Because~$R = 0$, and so~$R_i = 0$ for all~$i$, this is the number of integer solutions to
\begin{equation}
\label{eq:combination}
  \sum_{x \in \V} \,\xi (x) = M \quad \hbox{such that} \quad \xi (x) \geq 0 \quad \hbox{for all} \quad x \in \V,
\end{equation}
 which is also the number of sets of~$M$ items, each of which can be of~$N$ different types (combination with repetition).
 This quantity is known to be~($M + N - 1$ choose~$N - 1$).
\end{proof} \\ \\
 Using Lemma~\ref{lem:combination}, we now compute the number of configurations in the presence of banks.
\begin{lemma} --
\label{lem:combinatorics}
 The total number of configurations is
 $$ \begin{array}{rcl}
    \Lambda (N_i, R_i, M) & \n = \n &
    \displaystyle \sum_{a_1 = 0}^{R_1} \ \cdots \ \sum_{a_K = 0}^{R_K} \
    \displaystyle \sum_{b_1 = 0}^{N_1} \ \cdots \ \sum_{b_K = 0}^{N_K}
    \displaystyle {N_1 \choose b_1} \cdots {N_K \choose b_K} \vspace*{8pt} \\ && \hspace*{25pt}
    \displaystyle {a_1 - 1 \choose b_1 - 1} \cdots {a_K - 1 \choose b_K - 1}
    \displaystyle {M + a_1 + \cdots + a_K + N - b_1 - \cdots - b_K - 1 \choose N - b_1 - \cdots - b_K - 1}. \end{array} $$
\end{lemma}
\begin{proof}
 Letting~$S^+ \subset S$ be the subset of configurations where no one is in debt, i.e.,
 $$ S^+ = \{\xi \in S : \xi(x) \geq 0 \ \hbox{for all} \ x \in \V \}, $$
 condition~\eqref{eq:combination} is satisfies and it follows from~Lemma~\ref{lem:combination} that
\begin{equation}
\label{eq:cardS+}
  \card (S^+) = \binom{M + N - 1}{N - 1}.
\end{equation}
 Now, let~$S^- = S \setminus S^+$.
 To compute~$\card (S^-)$, we denote by
 $$ \phi (a_i, b_i, M) = \phi ((a_1, \ldots, a_K), (b_1, \ldots, b_K), M) $$
 the number of configurations with~$a_i$ coins borrowed from bank~$i$ and~$b_i$ individuals from bank~$i$ in debt.
 Then, summing over all the possible choices, we get
\begin{equation}
\label{eq:cardS-}
  \card (S^-) = \sum_{a_1 = 1}^{R_1} \cdots \sum_{a_K = 1}^{R_K} \
                \sum_{b_1 = 1}^{a_1 \land (N_1 - 1)} \cdots \sum_{b_K = 1}^{a_K \land (N_K - 1)} \phi ((a_1, \ldots, a_K), (b_1, \ldots, b_K), M).
\end{equation}
 In addition, for each bank~$i$,
\begin{equation}
\label{eq:combinatorics-1}
  \hbox{\# ways to choose the individuals in debt = ($N_i$ choose $b_i$)}.
\end{equation}
 Letting~$x_1, x_2, \ldots, x_{b_i}$ be these individuals, $A_k = - \xi (x_k)$ and~$B_k = - \xi (x_k) - 1$, the number of ways to allocate the debt among these individuals is
\begin{equation}
\label{eq:combinatorics-2}
  \begin{array}{l}
  \hbox{\# integer solutions to~$A_1 + \cdots + A_{b_i} = a_i$ with~$A_k > 0$} \vspace*{4pt} \\ \hspace*{20pt} = \
  \hbox{\# integer solutions to~$B_1 + \cdots + B_{b_i} = a_i - b_i$ with~$B_k \geq 0$} \vspace*{8pt} \\ \hspace*{50pt} = \
  \displaystyle \binom{a_i - b_i + b_i - 1}{b_i - 1} = \binom{a_i - 1}{b_i - 1} \end{array}
\end{equation}
 as in the proof of Lemma~\ref{lem:combination}.
 After choosing the total debt and the number of individuals in debt and following the same reasoning, we get that the number of ways to distribute the
 remaining~$M + a_1 + \cdots + a_K$ coins among the other~$N_1 + \cdots + N_K - b_1 - \cdots - b_K$ individuals is
\begin{equation}
\label{eq:combinatorics-3}
  \binom{(M + a_1 + \cdots + a_K) + (N_1 + \cdots + N_K - b_1 - \cdots - b_K) - 1}{(N_1 + \cdots + N_K - b_1 - \cdots - b_K) - 1}.
\end{equation}
 Combining~\eqref{eq:combinatorics-1}--\eqref{eq:combinatorics-3}, we deduce that
\begin{equation}
\label{eq:phi}
  \begin{array}{rcl}
  \phi (a_i, b_i, M) & \n = \n &
  \displaystyle {N_1 \choose b_1} \cdots {N_K \choose b_K} {a_1 - 1 \choose b_1 - 1} \cdots {a_K - 1 \choose b_K - 1} \vspace*{12pt} \\ && \hspace*{20pt} \times \,
  \displaystyle {M + a_1 + \cdots + a_K + N - b_1 - \cdots - b_K - 1 \choose N - b_1 - \cdots - b_K - 1}. \end{array}
\end{equation}
 Using~$\Lambda (N_i, R_i, M) = \card (S^+) + \card (S^-)$, \eqref{eq:cardS+}, \eqref{eq:cardS-}, \eqref{eq:phi}, and the convention
 $$ \binom{-1}{k} = 1 \quad \hbox{and} \quad \binom{n}{k} = 0 \quad \hbox{when} \quad k < 0 \ \hbox{or} \ k > n, $$
 we conclude that~$\Lambda (N_i, R_i, M)$ is equal to
 $$ \begin{array}{l}
    \displaystyle \binom{M + N - 1}{N - 1} + \sum_{a_1 = 1}^{R_1} \cdots \sum_{a_K = 1}^{R_K} \
    \displaystyle \sum_{b_1 = 1}^{a_1 \land (N_1 - 1)} \cdots \sum_{b_K = 1}^{a_k \land (N_k - 1)} \phi (a_i, b_i, M) \vspace*{4pt} \\ \hspace*{20pt} =
    \displaystyle \binom{M + N - 1}{N - 1} + \sum_{a_1 = 1}^{R_1} \cdots \sum_{a_K = 1}^{R_K} \ \sum_{b_1 = 1}^{a_1 \land (N_1 - 1)} \cdots \sum_{b_K = 1}^{a_k \land (N_k - 1)}
    \displaystyle {N_1 \choose b_1} \cdots {N_K \choose b_K} \vspace*{4pt} \\ \hspace*{60pt}
    \displaystyle {a_1 - 1 \choose b_1 - 1} \cdots {a_K - 1 \choose b_K - 1} {M + a_1 + \cdots + a_K + N - b_1 - \cdots - b_K - 1 \choose N - b_1 - \cdots - b_K - 1} \vspace*{8pt} \\ \hspace*{20pt} =
    \displaystyle \binom{M + N - 1}{N - 1} + \sum_{a_1 = 1}^{R_1} \cdots \sum_{a_K = 1}^{R_K} \ \sum_{b_1 = 1}^{N_1} \cdots \sum_{b_K = 1}^{N_k}
    \displaystyle {N_1 \choose b_1} \cdots {N_K \choose b_K} \vspace*{4pt} \\ \hspace*{60pt}
    \displaystyle {a_1 - 1 \choose b_1 - 1} \cdots {a_K - 1 \choose b_K - 1} {M + a_1 + \cdots + a_K + N - b_1 - \cdots - b_K - 1 \choose N - b_1 - \cdots - b_K - 1} \vspace*{8pt} \\ \hspace*{20pt} =
    \displaystyle \sum_{a_1 = 0}^{R_1} \cdots \sum_{a_K = 0}^{R_K} \ \sum_{b_1 = 0}^{N_1} \cdots \sum_{b_K = 0}^{N_k}
    \displaystyle {N_1 \choose b_1} \cdots {N_K \choose b_K} \vspace*{8pt} \\ \hspace*{60pt}
    \displaystyle {a_1 - 1 \choose b_1 - 1} \cdots {a_K - 1 \choose b_K - 1} {M + a_1 + \cdots + a_K + N - b_1 - \cdots - b_K - 1 \choose N - b_1 - \cdots - b_K - 1}. \end{array} $$
 This completes the proof.
\end{proof} \\ \\
 To deduce Theorem~\ref{th:combinatorics}, assume that~$x \in \V_j$ and observe that, given that there are~$c$ coins at vertex~$x$, there is a total of~$M - c$ coins scattered across the rest of the connected graph.
 In case~$x$ is not in debt, the customers of bank~$j$ other than~$x$ can altogether borrow~$R_j$ coins therefore the number of configurations such that~$\xi (x) = c \geq 0$ is given by
\begin{equation}
\label{eq:number-1}
\begin{array}{l}
  \Lambda ((N_1, \ldots, N_{j - 1}, N_j - 1, N_{j+ 1}, \ldots,  N_K), (R_1, \ldots, R_K), M - c) \vspace*{4pt} \\ \hspace*{120pt} =
  \Lambda ((N_1, \ldots, N_K) - e_j, (R_1, \ldots, R_K), M - c). \end{array}
\end{equation}
 In case~$x$ is in debt, the customers of bank~$j$ other than~$x$ can altogether borrow~$R_j + c$ coins therefore the number of configurations such that~$\xi (x) = c < 0$ is given by
\begin{equation}
\label{eq:number-2}
\begin{array}{l}
  \Lambda ((N_1, \ldots, N_K) - e_j, (R_1, \ldots, R_{j - 1}, R_j + c, R_{j + 1}, \ldots, R_K), M - c) \vspace*{4pt} \\ \hspace*{100pt} =
  \Lambda ((N_1, \ldots, N_K) - e_j, (R_1, \ldots, R_K) + c e_j, M - c). \end{array}
\end{equation}
 The theorem follows from combining~\eqref{eq:number-1}--\eqref{eq:number-2} and Lemma~\ref{lem:combination}, and using that the stationary distribution~$\pi$ is uniform on the set of configurations.

%%%%%%%%%%%%%%%%%%%%%%%%%%%%%%%%%%%%%%%%%%%%%%%%%%%%%%%%%%%%%%%%%%%%%%%%%%%%%%%%%%%%%%%%%%%%%%%%%%%%%%%%%%%%%%%%%%%%%%%%%%%%%%%%%%%%%%%%%%%%%%%%%%%%%%%%%%%%%%%%%%%%%%%%%%%%%%%%%%%%%%%%%%%%%%%%%%

\section{Convergence to the Laplace distribution}
\label{sec:laplace}
 This section is devoted to the proof of Theorem~\ref{th:laplace} which focuses on the distribution of money in the large population/temperature limit in the symmetric case where
\begin{equation}
\label{eq:symmetry}
  N_i = \card (\V_i) = N / K \quad \hbox{and} \quad R_i = R / K \quad \hbox{for all} \quad i = 1, 2, \ldots, K.
\end{equation}
 The first ingredient to prove the theorem is to observe that, according to Lemma~\ref{lem:uniform}, the stationary distribution of the process is the uniform distribution on the set of configurations \emph{regardless of the topology of the graph}.
 In particular, the limiting distribution of money on any connected graph is the same as the limiting distribution of money on the complete graph, in which case the evolution is described, in the large population limit, by a system of coupled differential equations for the fractions of individuals with a given number of coins~(see~Lemma~\ref{lem:mean-field}).
 Even though this system does not properly describes the process on arbitrary connected graphs, its limits as time goes to infinity corresponds to the limiting distribution of money of any large connected graph.
 The unique fixed point of the system of differential equations shows that, in the large population limit, the distribution of money converges to an asymmetric Laplace distribution~(see~Lemma~\ref{lem:fixed-point}), which is characterized by three parameters.
 The second ingredient is to use that the total number of coins across the population is preserved by the dynamics and that the number of coins in each bank is negligible compared to~$R_i$~(see Lemma~\ref{lem:bank-down}), to identify these three parameters. \\
\indent From now on, we assume that~$\G$ is the complete graph on~$N$ vertices.
 Because in this case the location of the individuals is unimportant, the process that only keeps track of the number of individuals with a given number of coins along with the bank they go to,
 $$ X_t = (X_t (i, c) : 1 \leq i \leq K, \ - R \leq c \leq M + R) \quad \hbox{where} \quad X_t (i, c) = \sum_{x \in \V_i} \,\ind \{\xi_t (x) = c \}, $$
 is a well-defined Markov chain.
 To describe the dynamics, let
 $$ X_t^- (i) = \sum_{c = -R/K}^{-1} X_t (i, c), \quad X_t^+ (i) = \sum_{c = 1}^{M + R} X_t (i, c), \quad B_t (i) = R/K + \sum_{c = -R/K}^{-1} c X_t (i, c) $$
 be the number of customers of bank~$i$ in debt, the number of customers of bank~$i$ with at least one coin, and the number of coins in bank~$i$ at time~$t$.
 Rescaling by~$N$, we also define
 $$ u_{i, c} (t) = \frac{X_t (i, c)}{N}, \quad u_i^{\pm} (t) = \frac{X_t^{\pm} (i)}{N} \quad \hbox{and} \quad p_i = E (\ind \{B_t (i) > 0 \}) = P (B_t (i) > 0) $$
 be respectively the fraction of individuals who are customers of bank~$i$ and have~$c$ coins, the fraction of individuals who are customers of bank~$i$ and have at least one coin/are in debt, and the probability that there is at least one coin in bank~$i$ at time~$t$.
 We also let
 $$ X_t^{\pm} = \sum_{i = 1}^K \,X_t^{\pm} (i) \quad \hbox{and} \quad u^{\pm} (t) = \sum_{i = 1}^K \,u_i^{\pm} (t) = \frac{X_t^{\pm}}{N} $$
 be the number and the fraction of individuals with at least one coin/in debt.
 Then, the functions~$u_{i, c}$ satisfy the following system of coupled differential equations.
\begin{lemma} --
% [mean-field] --
\label{lem:mean-field}
 In the limit as~$N \to \infty$,
\begin{equation}
\label{eq:mean-field-0}
  u_{i, c}' = \left\{\begin{array}{lcl}
              \bar u \,(u_{i, c - 1} - u_{i, c}) - (u_{i, c} - u_{i, c + 1}) & \hbox{for all} & c > 0 \vspace*{4pt} \\
              \bar u \,(u_{i, c - 1} - u_{i, c}) - p_i (u_{i, c} - u_{i, c + 1}) & \hbox{for all} & c < 0 \end{array} \right.
\end{equation}
 where~$\bar u = u^+ + p_1 (u_{1, 0} + u_1^-) + \cdots + p_K (u_{K, 0} + u_K^-)$.
\end{lemma}
\begin{proof}
 Note that the number of individuals who can give a coin (because they already have one or because they can borrow one from their bank) can be expressed as
\begin{equation}
\label{eq:mean-field-1}
  \bar X_t = X_t^+ + \sum_{i = 1}^K \ (X_t (i, 0) + X_t^- (i)) \,\ind \{B_t (i) > 0 \}.
\end{equation}
 Using that the number of customers of bank~$i$ with~$c$ coins increases by one when
\begin{itemize}
 \item a vertex in~$\V_i$ with~$c - 1$ coins gets a coin from a vertex not in~$\V_i$ with~$c$ coins or \vspace*{2pt}
 \item a vertex in~$\V_i$ with~$c + 1$ coins gives a coin to a vertex not in~$\V_i$ with~$c$ coins,
\end{itemize}
 we deduce that, for all~$c > 0$,
\begin{equation}
\label{eq:mean-field-2}
  \begin{array}{rcl}
    P (X_{t + 1} (i, c) = X_t (i, c) + 1 \,| \,X_t) & \n = \n &
  \displaystyle \bigg(\frac{\bar X_t}{N} - \frac{X_t (i, c)}{N} \bigg) \,\frac{X_t (i, c - 1)}{N} \vspace*{8pt} \\ && \hspace*{20pt} + \
  \displaystyle \frac{X_t (i, c + 1)}{N} \,\bigg(1 - \frac{X_t (i, c)}{N} \bigg) \end{array}
\end{equation}
 whereas, for all~$c < 0$,
\begin{equation}
\label{eq:mean-field-3}
  \begin{array}{rcl}
    P (X_{t + 1} (i, c) = X_t (i, c) + 1 \,| \,X_t) & \n = \n &
  \displaystyle \bigg(\frac{\bar X_t}{N} - \frac{X_t (i, c) \,\ind \{B_t (i) > 0 \}}{N} \bigg) \,\frac{X_t (i, c - 1)}{N} \vspace*{8pt} \\ && \hspace*{20pt} + \
  \displaystyle \frac{X_t (i, c + 1) \,\ind \{B_t (i) > 0 \}}{N} \,\bigg(1 - \frac{X_t (i, c)}{N} \bigg). \end{array}
\end{equation}
 Similarly, the number of customers of bank~$i$ with~$c$ coins decreases by one when
\begin{itemize}
 \item a vertex in~$\V_i$ with~$c$ coins gets a coin from a vertex not in~$\V_i$ with~$c + 1$ coins or \vspace*{2pt}
 \item a vertex in~$\V_i$ with~$c$ coins gives a coin to a vertex not in~$\V_i$ with~$c - 1$ coins
\end{itemize}
 therefore, for all~$c > 0$,
\begin{equation}
\label{eq:mean-field-4}
  \begin{array}{rcl}
    P (X_{t + 1} (i, c) = X_t (i, c) - 1 \,| \,X_t) & \n = \n &
  \displaystyle \bigg(\frac{\bar X_t}{N} - \frac{X_t (i, c + 1)}{N} \bigg) \,\frac{X_t (i, c)}{N} \vspace*{8pt} \\ && \hspace*{20pt} + \
  \displaystyle \frac{X_t (i, c)}{N} \,\bigg(1 - \frac{X_t (i, c - 1)}{N} \bigg) \end{array}
\end{equation}
 whereas, for all~$c < 0$,
\begin{equation}
\label{eq:mean-field-5}
  \begin{array}{rcl}
    P (X_{t + 1} (c) = X_t (c) - 1 \,| \,X_t) & \n = \n &
  \displaystyle \bigg(\frac{\bar X_t}{N} - \frac{X_t (i, c + 1) \,\ind \{B_t (i) > 0 \}}{N} \bigg) \,\frac{X_t (i, c)}{N} \vspace*{8pt} \\ && \hspace*{20pt} + \
  \displaystyle \frac{X_t (i, c) \,\ind \{B_t (i) > 0 \}}{N} \,\bigg(1 - \frac{X_t (i, c - 1)}{N} \bigg). \end{array}
\end{equation}
 In the limit as~$N \to \infty$, the process is described by a system of differential equations for the densities~$u_{i, c}$.
 Taking the expected value in~\eqref{eq:mean-field-1}, we get
 $$ \begin{array}{rcl}
    \displaystyle E \,\bigg(\frac{\bar X_t}{N} \bigg) & \n = \n &
    \displaystyle E \,\bigg(\frac{X_t^+}{N} + \sum_{i = 1}^K \bigg(\frac{X_t (i, 0)}{N} + \frac{X_t^- (i)}{N} \bigg) \,\ind \{B_t (i) > 0 \} \bigg) \vspace*{4pt} \\ & \n = \n &
    \displaystyle u^+ + \sum_{i = 1}^K \, p_i (u_{i, 0} + u_i^-) = \bar u. \end{array} $$
 In addition, combining~\eqref{eq:mean-field-2} and~\eqref{eq:mean-field-4}, we deduce that, in the limit as~$N \to \infty$,
 $$ \begin{array}{rcl}
      u_{i, c}' & \n = \n & (\bar u - u_{i, c}) \,u_{i, c - 1} + u_{i, c + 1} (1 - u_{i, c}) - (\bar u - u_{i, c + 1}) \,u_{i, c} - u_{i, c} (1 - u_{i, c - 1}) \vspace*{4pt} \\
                & \n = \n &  \bar u \,(u_{i, c - 1} - u_{i, c}) - (u_{i, c} - u_{i, c + 1}) \end{array} $$
 for all integers~$c > 0$, while combining~\eqref{eq:mean-field-3} and~\eqref{eq:mean-field-5} gives that
 $$ \begin{array}{rcl}
      u_{i, c}' & \n = \n & (\bar u - p_i u_{i, c}) \,u_{i, c - 1} + p_i u_{i, c + 1} (1 - u_{i, c}) - (\bar u - p_i u_{i, c + 1}) \,u_{i, c} - p_i u_{i, c} (1 - u_{i, c - 1}) \vspace*{4pt} \\
                & \n = \n &  \bar u \,(u_{i, c - 1} - u_{i, c}) - p_i (u_{i, c} - u_{i, c + 1}). \end{array} $$
 for all integers~$c < 0$.
\end{proof} \\ \\
 Letting~$u_c (t)$ be the total fraction of individuals with~$c$ coins~(regardless of their bank), the distribution of money at equilibrium in the large population limit is described by
 $$ f (c) = \lim_{t \to \infty} \lim_{N \to \infty} P (\xi_t (x) = c) = \lim_{t \to \infty} \lim_{N \to \infty} u_c (t) = \sum_{i = 1}^K \ \lim_{t \to \infty} \lim_{N \to \infty} u_{i, c} (t). $$
 The next lemma shows that~$f$ is an asymmetric Laplace distribution.
\begin{lemma} --
\label{lem:fixed-point}
 Assume~\eqref{eq:symmetry}.
 Then, there exist~$\mu, a, b > 0$ such that
 $$ f (c) = \left\{\hspace*{-3pt} \begin{array}{lcl} \mu \,e^{- ac} & \hbox{for} & c \geq 0 \vspace*{3pt} \\
                                                     \mu \,e^{+ bc} & \hbox{for} & c \leq 0. \end{array} \right. $$
\end{lemma}
\begin{proof}
 The assumption~\eqref{eq:symmetry} implies that, at  least in the large population limit and starting from the configuration where all the individuals have exactly~$T$ coins,
 $$ u_c (t) = K u_{1, c} (t) = \cdots = K u_{K, c} (t) \quad \hbox{and} \quad p_1 = \cdots = p_K = p. $$
 In particular, summing~\eqref{eq:mean-field-0} over~$i = 1, 2, \ldots, K$, we get
\begin{equation}
\label{eq:fixed-point}
  u_c' = \left\{\begin{array}{lcl}
         \bar u \,(u_{c - 1} - u_c) -   (u_c - u_{c + 1}) & \hbox{for all} & c > 0 \vspace*{4pt} \\
         \bar u \,(u_{c - 1} - u_c) - p (u_c - u_{c + 1}) & \hbox{for all} & c < 0 \end{array} \right.
\end{equation}
 where~$\bar u = u^+ + p (u_0 + u^-)$.
 Letting~$v_c = u_c - u_{c + 1}$, we can rewrite~\eqref{eq:fixed-point} as
 $$ u_c' = \left\{\begin{array}{lcl}
           \bar u \,v_{c - 1} -   v_c & \hbox{for all} & c > 0 \vspace*{4pt} \\
           \bar u \,v_{c - 1} - p v_c & \hbox{for all} & c < 0 \end{array} \right. $$
 In particular, setting~$u_c' = 0$ to find the fixed point, we get
 $$ v_c = \left\{\begin{array}{l}
          \bar u \,v_{c - 1} = \cdots = \bar u^c v_0 = \bar u^c (u_0 - u_1) \quad \hbox{for all} \quad c > 0 \vspace*{5pt} \\
         (p / \bar u) \,v_{c + 1} = \cdots = (p / \bar u)^{- c - 1} v_{-1} = (p / \bar u)^{- c - 1} (u_{-1} - u_0) \quad \hbox{for all} \quad c < 0 \end{array} \right. $$
 from which it follows that
 $$ \begin{array}{l}
    \displaystyle u_c = u_0 - \sum_{n = 0}^{c - 1} \,v_n = u_0 - \sum_{n = 0}^{c - 1} \,\bar u^n \,v_0 = u_0 - \frac{1 - \bar u^c}{1 - \bar u} \ (u_0 - u_1)
    \quad \hbox{for} \quad c > 0 \vspace*{4pt} \\
    \displaystyle u_c = u_0 + \sum_{n = 1}^{- c} \,v_{- n} = u_0 + \sum_{n = 1}^{- c} \,\bigg(\frac{p}{\bar u} \bigg)^{n - 1} \,v_{-1} = u_0 - \frac{1 - (p / \bar u)^{-c}}{1 - p / \bar u} \ (u_0 - u_{-1})
    \quad \hbox{for} \quad c < 0. \end{array} $$
 Finally, because~$u_c \to 0$ as~$c \to \pm \infty$, we must have
 $$ u_c = u_0 \,\bar u^c \quad \hbox{for all} \quad c > 0 \qquad \hbox{and} \qquad u_c = u_0 (p / \bar u)^c \quad \hbox{for all} \quad c < 0, $$
 showing that the lemma holds for the parameters
 $$ \begin{array}{rcl}
      \mu & \n = \n & u_0 = \lim_{t \to \infty} u_0 (t) > 0 \vspace*{4pt} \\
        a & \n = \n & - \ln (\bar u) = \lim_{t \to \infty} (- \ln (\bar u (t))) > 0 \vspace*{4pt} \\
        b & \n = \n & - \ln (\bar u / p) = \lim_{t \to \infty} (- \ln (\bar u (t) / p (t))) > 0. \end{array} $$
 This completes the proof.
\end{proof} \\ \\
 To complete the proof of the theorem, the next step is to find explicit expressions of the parameters~$\mu, a, b$, which will be done by deriving a system of three equations involving the parameters.
 The first two equations are obtained in the next lemma by using that~$f$ is a density function and that the money temperature is preserved by the dynamics.
\begin{lemma} --
\label{lem:density-mean}
 For all~$M$ and~$R = \rho M$, we have
\begin{equation}
\label{eq:density-mean}
  \frac{\mu}{a} + \frac{\mu}{b} = 1 \quad \hbox{and} \qquad \frac{\mu}{a^2} - \frac{\mu}{b^2} = \frac{M}{N} = T.
\end{equation}
\end{lemma}
\begin{proof}
 To begin with, observe that
 $$ \int_{- \infty}^{\infty} f (c) \,dc = \int_0^{\infty} \mu \,e^{- ac} \,dc + \int_{- \infty}^0 K_ 0 \,e^{+ bc} \,dc = \frac{\mu}{a} + \frac{\mu}{b}. $$
 Because the function~$f$ is a density function, the integral above must be equal to one, which gives the first equation in the lemma.
 Note also that
 $$ \begin{array}{l}
    \displaystyle \int_{- \infty}^{\infty} c f (c) \,dc = \int_0^{\infty} \mu \,c e^{- ac} \,dc + \int_{- \infty}^0 \mu \,c e^{+ bc} \,dc \vspace*{8pt} \\ \hspace*{80pt} =
    \displaystyle \int_0^{\infty} \frac{\mu \,c e^{- ac}}{a} \,dc - \int_{- \infty}^0 \frac{\mu \,c e^{+ bc}}{b} \,dc = \frac{\mu}{a^2} - \frac{\mu}{b^2}. \end{array} $$
 Because the system is conservative and the money temperature is preserved by the dynamics, the integral above, which represents the mean number of coins per individual at equilibrium, must be equal to the money temperature~$T = M/N$.
 This gives the second equation.
\end{proof} \\ \\
 The third equation involving~$\mu, a, b$ expresses the fact that, at equilibrium, the total number of coins shared by the individuals not in debt is of the order of~$M + R$, the total number of coins in the system, which holds because the number of coins in each of the banks becomes much smaller than~$R$ in the large population limit.
 The number of coins in the banks has a negative drift of the order of~$u_0 (t)$ so, to show that the number of coins in the banks indeed becomes relatively small, the next step is to prove that, at equilibrium, the fraction of individuals with exactly zero coin is not too small, of the order of the reciprocal of the money temperature.
\begin{lemma} --
\label{lem:origin}
 There exists~$C_0 > 0$ such that~$\mu = C_0/T$.
\end{lemma}
\begin{proof}
 Let~$\delta = b/a > 1$.
 Using the second then the first equation in~\eqref{eq:density-mean},
 $$ T = \frac{\mu}{a^2} - \frac{\mu}{b^2} = \mu \bigg(\frac{1}{a} + \frac{1}{b} \bigg) \bigg(\frac{1}{a} - \frac{1}{b} \bigg) = \frac{1}{a} - \frac{1}{b} = \frac{1}{a} - \frac{1}{\delta a} = \bigg(1 - \frac{1}{\delta} \bigg) \frac{1}{a} $$
 from which it follows, using again the first equation in~\eqref{eq:density-mean}, that
 $$ a = \bigg(1 - \frac{1}{\delta} \bigg) \frac{1}{T} \quad \hbox{and} \quad \mu = \frac{1}{1/a + 1/b} = \frac{(1 - 1/\delta)(\delta - 1)}{(1 - 1/\delta) + (\delta - 1)} \ \frac{1}{T} = \frac{(\delta - 1)^2}{\delta^2 - 1} \ \frac{1}{T}. $$
 This shows that the lemma holds for~$C_0 = (\delta - 1)^2 / (\delta^2 - 1) > 0$.
\end{proof} \\ \\
 Using Lemma~\ref{lem:origin}, we can now prove that, at equilibrium, the number of coins in each of the banks is much smaller than~$R$, the initial number of coins in the banks.
\begin{lemma} --
\label{lem:bank-down}
 In the limit as~$N \to \infty$, $\lim_{t \to \infty} B_t (i) / R \to 0$.
\end{lemma}
\begin{proof}
 The number of coins in bank~$i$ increases by one each time an individual with at least one coin gives one coin to a customer of bank~$i$ who is in debt therefore
\begin{equation}
\label{eq:bank-down-1}
  P (B_{t + 1} (i) = B_t (i) + 1 \,| \,X_t) = u^+ (t) \,u_i^- (t).
\end{equation}
 Similarly, the number of coins in bank~$i$ decreases by one each time a customer of bank~$i$ with zero coin or in debt gives one coin to an individual not in debt therefore
\begin{equation}
\label{eq:bank-down-2}
  P (B_{t + 1} (i) = B_t (i) - 1 \,| \,X_t, B_t (i) > 0) = (u_i^- (t) + u_{i, 0} (t))(u^+ (t) + u_0 (t)).
\end{equation}
 The transition probabilities~\eqref{eq:bank-down-1} and~\eqref{eq:bank-down-2} show that, once there is at least one individual in debt, the process~$(B_t (i))$ is dominated by a one-dimensional symmetric random walk with reflecting boundary at zero so the process reaches zero after a time of the order of at most~$R^2$.
 We now look at the time it takes for bank~$i$ to return away from state zero.
 It follows from the transition probabilities~\eqref{eq:bank-down-1} and~\eqref{eq:bank-down-2} that, in the large population limit and at equilibrium,
 $$ \begin{array}{rcl}
      P (B_{t + 1} (i) = B_t (i) + 1 \,| \,X_t) & \n = \n & u^+ \,u_i^- = u^+ \,u^- / K \vspace*{4pt} \\
      P (B_{t + 1} (i) = B_t (i) - 1 \,| \,X_t, B_t (i) > 0) & \n = \n & (u^- / K + u_0 / K)(u^+ + u_0) = u^+ \,u^- / K + u_0 / K. \end{array} $$
 Because~$\mu = C_0/T$ according to Lemma~\ref{lem:origin}, near the stationary distribution,
 $$ u_0 \geq \frac{\mu}{2} = \frac{C_0}{2T} \qquad \hbox{and} \qquad \frac{u^+ \,u^- / K + u_0 / K}{u^+ \,u^- / K} \geq 1 + 4 u_0 \geq 1 + \frac{2 C_0}{T}. $$
 In particular, letting~$\tau = \inf \{t : B_t (i) = 0 \ \hbox{or} \ B_t (i) > \sqrt{R_i T} \}$, recalling~\eqref{eq:bank-down-1} and~\eqref{eq:bank-down-2}, and applying the optional stopping theorem, we deduce that
 $$ \begin{array}{l}
      P (B_{\tau} (i) > \sqrt{R_i T} \,| \,B_0 (i) = 1) \leq
    \displaystyle \frac{1 - (1 + 2 C_0/T)}{1 - (1 + 2 C_0/T)^{\sqrt{R_i T}}} \leq \bigg(\frac{1}{1 + 2 C_0/T} \bigg)^{\sqrt{R_i T} - 1} \vspace*{8pt} \\ \hspace*{80pt} \leq
    \displaystyle \bigg(1 - \frac{C_0}{T} \bigg)^{\sqrt{R_i T} - 1} \approx \exp \bigg(- C_0 \,\sqrt{\frac{R_i}{T}} \bigg) = \exp (- C_0 \sqrt{\rho_i N}). \end{array} $$
 This, together with the strong Markov property, implies that the expected amount of time for the number of coins in bank~$i$ to return above the threshold~$\sqrt{R_i T}$ is larger than
 $$ E \bigg(\geometric \Big(\exp (- C_0 \sqrt{\rho_i N}) \Big) \bigg) = \exp (C_0 \sqrt{\rho_i N}). $$
 Observing also that~$\sqrt{R_i T} / R \to 0$ in the limit as~$N \to \infty$, the result follows.
\end{proof} \\ \\
 Using the previous lemma, we can now derive a third equation involving the parameters of the asymmetric Laplace distribution.
\begin{lemma} --
\label{lem:positive-mean}
 For all~$M$ and~$R = \rho M$, we have
\begin{equation}
\label{eq:positive-mean}
  \frac{\mu}{a^2} = \frac{M + R}{N} = (1 + \rho) T.
\end{equation}
\end{lemma}
\begin{proof}
 According to Lemma~\ref{lem:bank-down}, in the large population limit and at equilibrium, the total number of coins in all the banks is much smaller than~$R$, so the number of coins shared by all the individuals not in debt converges to~$M + R$, i.e.,
\begin{equation}
\label{eq:positive-mean-1}
 \lim_{t \to \infty} \lim_{N \to \infty} \frac{1}{N} \sum_{x \in \V} \xi_t (x) \,\ind \{\xi_t (x) > 0 \} = \lim_{N \to \infty} \frac{M + R}{N} = (1 + \rho) T.
\end{equation}
 Because the left-hand side of~\eqref{eq:positive-mean-1} is also equal to
 $$ \int_0^{\infty} c f (c) \,dc = \int_0^{\infty} \mu \,c e^{- ac} \,dc = \int_0^{\infty} \frac{\mu \,c e^{- ac}}{a} \,dc = \frac{\mu}{a^2} $$
 the lemma follows.
\end{proof} \\ \\
 The final step to prove the theorem is to use the three equations in Lemmas~\ref{lem:density-mean} and~\ref{lem:positive-mean} to express the three parameters~$\mu, a, b$ of the asymmetric Laplace distribution as a function of the money temperature~$T$ and the fraction of coins from the banks~$\rho$.
\begin{lemma} --
\label{lem:parameters}
 For all~$M$ and~$R = \rho M$, we have
 $$ \mu = \frac{1}{T} \bigg(\sqrt{1 + \rho} - \sqrt{\rho} \bigg)^2, \qquad
      a = \frac{1}{T} \bigg(1 - \sqrt{\frac{\rho}{1 + \rho}} \bigg), \qquad
      b = \frac{1}{T} \bigg(\sqrt{\frac{1 + \rho}{\rho}} - 1 \bigg). $$
\end{lemma}
\begin{proof}
 Recall from Lemmas~\ref{lem:density-mean} and~\ref{lem:positive-mean} that
\begin{equation}
\label{eq:parameters-1}
  \frac{\mu}{a} + \frac{\mu}{b} = 1, \qquad \frac{\mu}{a^2} - \frac{\mu}{b^2} = T, \qquad \frac{\mu}{a^2} = (1 + \rho) T.
\end{equation}
 Subtracting the second equation from the third equation in~\eqref{eq:parameters-1}, then taking the square root of the ratio with the third equation, we get
\begin{equation}
\label{eq:parameters-2}
  \frac{\mu}{b^2} = \frac{\mu}{a^2} - \bigg(\frac{\mu}{a^2} - \frac{\mu}{b^2} \bigg) = \rho T \quad \hbox{and} \quad
  \frac{a}{b} = \sqrt{\frac{\mu}{b^2}} \bigg/ \sqrt{\frac{\mu}{a^2}} = \sqrt{\frac{\rho}{1 + \rho}}.
\end{equation}
 Using the first two equations in~\eqref{eq:parameters-1}, we also have
\begin{equation}
\label{eq:parameters-3}
  \frac{1}{a} - \frac{1}{b} = \bigg(\frac{\mu}{a} + \frac{\mu}{b} \bigg) \bigg(\frac{1}{a} - \frac{1}{b} \bigg) = \frac{\mu}{a^2} - \frac{\mu}{b^2} = T.
\end{equation}
 From~\eqref{eq:parameters-3} and the last equation in~\eqref{eq:parameters-2}, we deduce that
 $$ \begin{array}{rclrcl}
    \displaystyle \frac{1}{a} - \frac{1}{b} = \frac{1}{a} \bigg(1 - \sqrt{\frac{\rho}{1 + \rho}} \bigg) = T & \hbox{so} &
    \displaystyle a = \frac{1}{T} \bigg(1 - \sqrt{\frac{\rho}{1 + \rho}} \bigg) \vspace*{10pt} \\
    \displaystyle \frac{1}{a} - \frac{1}{b} = \frac{1}{b} \bigg(\sqrt{\frac{1 + \rho}{\rho}} - 1\bigg) = T & \hbox{so} &
    \displaystyle b = \frac{1}{T} \bigg(\sqrt{\frac{1 + \rho}{\rho}} - 1 \bigg). \end{array} $$
 Finally, using the last equation in~\eqref{eq:parameters-1} and the expression of~$a$,
 $$ \mu = \frac{1 + \rho}{T} \bigg(1 - \sqrt{\frac{\rho}{1 + \rho}} \bigg)^2 =  \frac{1}{T} \bigg(\sqrt{1 + \rho} - \sqrt{\rho} \bigg)^2. $$
 This completes the proof.
\end{proof} \\ \\
 From the previous lemma, we can also deduce that, at equilibrium, the fraction of individuals with at least one coin and the fraction of individuals in debt are given by
 $$ \begin{array}{l}
    \displaystyle u^+ = \int_0^{\infty} \mu \,e^{-ac} \,dc = \frac{\mu}{a} =
    \frac{\sqrt{1 + \rho} \,(\sqrt{1 + \rho} - \sqrt{\rho})^2}{\sqrt{1 + \rho} - \sqrt{\rho}} = \sqrt{1 + \rho} \ (\sqrt{1 + \rho} - \sqrt{\rho}) \vspace*{8pt} \\
    \displaystyle u^- = \int_0^{\infty} \mu \,e^{+bc} \,dc = \frac{\mu}{b} =
    \frac{\sqrt{\rho} \,(\sqrt{1 + \rho} - \sqrt{\rho})^2}{\sqrt{1 + \rho} - \sqrt{\rho}} = \sqrt{\rho} \ (\sqrt{1 + \rho} - \sqrt{\rho}). \end{array} $$
 In addition, the fraction of time the bank is empty is
 $$ 1 - p = 1 - \frac{\bar u}{\bar u / p} = 1 - \frac{e^{-a}}{e^{-b}} = 1 - \exp \bigg(\frac{1}{T} \bigg(\sqrt{\frac{1 + \rho}{\rho}} + \sqrt{\frac{\rho}{1 + \rho}}    \bigg) \bigg). $$

%%%%%%%%%%%%%%%%%%%%%%%%%%%%%%%%%%%%%%%%%%%%%%%%%%%%%%%%%%%%%%%%%%%%%%%%%%%%%%%%%%%%%%%%%%%%%%%%%%%%%%%%%%%%%%%%%%%%%%%%%%%%%%%%%%%%%%%%%%%%%%%%%%%%%%%%%%%%%%%%%%%%%%%%%%%%%%%%%%%%%%%%%%%%%%%%%%

\end{document}